\documentclass[a4paper,12pt,thmsa]{amsart}
\usepackage[a4paper,marginratio={1:1},scale={0.72,0.74},footskip=7mm,headsep=10mm]{geometry}

\usepackage{amsfonts}
\usepackage{amssymb,amsmath,latexsym}
\usepackage[dvips]{graphics}
\usepackage{graphicx,subfigure}
\usepackage[dvips]{color}
\usepackage[active]{srcltx}
\usepackage{amsmath}
\usepackage{amsfonts}
\usepackage{amssymb}
\usepackage{psfrag}
\usepackage{color}
\usepackage{url}
\usepackage{amsthm}
\usepackage{array}
\usepackage{pst-tree}
\usepackage{lscape}
\usepackage{dsfont}
\usepackage[utf8]{inputenc}							
\usepackage[T1]{fontenc}   
\usepackage{pstricks,pstricks-add}
\usepackage{mathrsfs}  
\usepackage{bbm}
\usepackage{tikz}
\usepackage{hyperref}

\def\p{\mathbb{P}}
\def\e{\mathbb{E}}
\def\ind{\mbox{\rm 1\hspace{-0.04in}I}}

\theoremstyle{plain}
\newtheorem{theorem}{Theorem}
\newtheorem{lemma}{Lemma}

\newtheorem{corollary}{Corollary}
\newtheorem{proposition}{Proposition}
\theoremstyle{definition}

\newtheorem{remark}{Remark}

\numberwithin{equation}{section}

\title[L\'evy processes resurrected in the positive half-line]
{L\'evy processes resurrected in the positive half-line}

\author{Mar\'\i a Emilia Caballero}

\address{Mar\'\i a Emilia Caballero, Instituto de Matem\'aticas, Universidad Nacional Aut\'onoma de M\'exico. 
Distrito federal CP 04510, M\'exico.}

\email{mariaemica@gmail.com}

\author{Lo\"ic Chaumont}

\address{Lo\"ic Chaumont, Univ Angers, CNRS, LAREMA, SFR MATHSTIC, F-49000 Angers, France}

\email{loic.chaumont@univ-angers.fr}

\author{V{\'\i }ctor Rivero}

\address{V{\'\i }ctor Rivero, Centro de
Investigaci\'on en Matem\'aticas (CIMAT A.C.). Calle Jalisco s/n, 36240 Guanajuato, Guanajuato M\'exico. }

\email{rivero@cimat.mx}

\keywords{L\'evy process; killing; resurrecting; conservativeness; creeping; subordinator}

\subjclass[2010]{60G51}

\date{\today}

\begin{document}

\begin{abstract}
A L\'evy processes resurrected in the positive half-line is a Markov process obtained by removing 
successively all jumps that make it negative. A natural question, given this construction, is whether 
the resulting process is absorbed at 0 or not. We first describe the law of the resurrected process 
in terms of that of the initial L\'evy process. Then in many important classes of L\'evy processes,
we give conditions for absorption and conditions for non absorption bearing on the characteristics of the 
initial L\'evy process.
\end{abstract}

\maketitle

\section{Introduction}\label{sec:intro}

Let $X^{(1)}$ be a real L\'evy process starting from a nonnegative level. If $X^{(1)}$ becomes negative by 
a jump, then remove this jump and if it reaches 0 from above, then let the process be absorbed at 0. 
Then call $X^{(2)}$ the process thus obtained and apply to $X^{(2)}$ the same transformation as for $X^{(1)}$.
Call $X^{(3)}$ the new process, and so on. The process $Z$ obtained by repeating 
this procedure as long as $X^{(n)}$ crosses 0 by a jump is called the L\'evy process $X^{(1)}$ resurrected 
in the positive half-line. The level 0 is clearly absorbing for the resurrected process $Z$ and a natural 
question that will occupy much of this article is whether or not $Z$ hits 0 in a finite time.

The resurrected process $Z$ is actually  a special case of Markov process constructed by piecing out as first 
introduced by Ikeda, Nagasawa and Watanabe \cite{inw} and studied in more detail by Meyer \cite{me}. The problem 
of the finiteness of the lifetime (that is non conservativeness) of resurrected processes was first pointed out 
and studied in \cite{inw}, Proposition 4.3, see also Sato \cite{sa}, Theorem 4.5. These results are obtained in a 
{rather general setting but only allow us} to solve the case where the negative
half line is not regular for the L\'evy process $X^{(1)}$, see Proposition~\ref{infinite} below. Later, Bogdan, 
Burdzy and Chen \cite{cb} considered a similar question for multidimensional symmetric stable processes 
resurrected in open sets with finite Lebesgue measure. This work was then extended by Wagner \cite{wa} to any 
symmetric L\'evy process, see Theorem 2.6 therein. The papers \cite{cb} and \cite{wa} strongly bear on the 
symmetry of the process and the powerful tools provided by Dirichlet forms that can be used in this case. Then 
recently, Kim, Song and Vondra\v{c}ek \cite{ksv} tackled the problem of conservativeness for positive 
self-similar Markov processes resurrected in the positive half line. They provide a complete solution of the
problem in this case, which includes stable L\'evy processes, as a direct application of the Lamperti 
transformation, see Section \ref{5223} below. Finally, in the recent papers \cite{dtx} and \cite{li}, 
resurrection of non-increasing Lévy processes has been studied in the context of fractional derivatives, 
this is closely related to our Corollaries~\ref{3789} and \ref{3729}. 

The main objective of the present paper is to give conditions for conservativeness of a real L\'evy process 
resurrected in the positive half line. In the next section, we give a detailed definition of this process 
whose law is described in terms of that of the process killed when it reaches the negative half line. In 
particular, we specify the explicit form of the resurrection kernel. Then in Section \ref{6362} we prove that, 
when the initial L\'evy process $X^{(1)}$ creeps downward and satisfy certain additional condition, the 
resurrected process is absorbed at 0 with probability one, independently of its starting point. 
Some criteria for conservativeness (non absorption) are given in Section \ref{section:MC} and some criteria 
for absorption are given in Section \ref{8362}. This section actually contains the most delicate case, that is 
when $X^{(1)}$ enters immediately in the negative half line and drifts to $-\infty$. We give a sufficient 
condition for absorption in Theorem \ref{H-inf} but up to now, even when $X^{(1)}$ is the negative of a 
subordinator, we don't know whether this condition can be dropped or not. The stable case already mentioned 
above is treated in Section \ref{5223} and we address some perspectives in the last section.

We close this introduction by pointing out that even though we provide a rather large set of criteria to 
determine whether a resurrected process is conservative or not, there remain various open questions related 
to this. Thus, we see this paper as an invitation for a broader audience to work on the subject. {Many other interesting questions related to these processes remain to be solved, but the one treated here is a key one. Knowing when the process $Z$ is not conservative naturally invites {us} to study its recurrent extensions, i.e. the totality of processes $\widetilde{Z}$ for which $0$ is {a} regular and recurrent state, such that the process obtained by killing $\widetilde{Z}$ at its first hitting time of zero, has the same law as $Z.$ In this direction, some results are known in specific cases, see e.g. \cite{cb} for symmetric stable processes. The question is being studied by 
the authors in an ongoing research.}

\section{The resurrected process}\label{2456}

\subsection{Basic definition}\label{5221}
Let $X=(X_t)_{t\ge0}$ be any real L\'evy process starting from 0. The {\it resurrected L\'evy process} takes 
its name from the following recursive pathwise construction. Let $x\ge0$ and let $X^{(n)}$, $n\ge1,$ be the 
sequence of stochastic processes defined by $X^{(1)}=X+x$ and for $n\ge2$, if 
$\tau_{n-1}:=\inf\{t\ge0:X_t^{(n-1)}\le0\}<\infty$, 
then
\begin{equation}\label{8002}
X^{(n)}_t=
\left\{\begin{array}{ll}
X^{(n-1)}_t,\;\;&\mbox{if $t<\tau_{n-1}$},\\
X_t^{(n-1)}-{(X^{(n-1)}_{\tau_{n-1}}-X^{(n-1)}_{\tau_{n-1}-})},\;\;&\mbox{if $t\ge\tau_{n-1}$},
\end{array}\right.
\end{equation}
where $X_{0-}=0$ and $X^{(n)}=X^{(n-1)}$, if $\tau_{n-1}=\infty$. 
The process $X^{(n)}$ is obtained by removing from $X^{(n-1)}$ its first jump through 0. Note that if for 
some $n\ge1$, $X^{(n)}$ hits 0 continuously, that is $X_{\tau_{n}}=X_{\tau_{n}-}$, then 
$X^{(k)}=X^{(n)}$, for all $k\ge n$. Note also that $(\tau_{n})_{n\ge1}$ is a non decreasing sequence 
of random times and that $X^{(k)}_t=X^{(n)}_t$, for all $k\ge n$, whenever $t\le\tau_{n}$. 
This allows us to define for each $t\ge0$ the random variable 
$Z_t:=\lim_{n\rightarrow\infty}X^{(n)}_t\ind_{\{t<\tau_{n}\}}$. Then $Z=(Z_t)_{t\ge0}$ defines a c\`adl\`ag
stochastic process which is nonnegative, absorbed at 0 and satisfies $Z_0=x$, a.s. The first hitting time 
of 0 by $Z$ is obtained as the limit of the sequence {$(\tau_n, n\geq 0)$}, that is, 
\[\zeta:=\inf\{t:Z_t=0\}=\lim_{n\rightarrow\infty}\tau_{n}.\] 
We will also use the following more synthetic expression of $Z$, 
\begin{equation}\label{8633}
Z_t=\sum_{n\ge1} X^{(n)}_t\ind_{\{\tau_{n-1}\le t<\tau_{n}\}}\,,\;\;\;t\ge0\,,
\end{equation}
where we set $\tau_0=0$. Then the process $Z$ is called the resurrected L\'evy process $X$ starting from 
$x$.\\

This process is actually a special case of constructing a Markov process by piecing out, as described in 
\cite{inw} and \cite{me}. More specifically, let $\p_x$, $x\in\mathbb{R}$ be a family of probability measures 
under which $X$ is a L\'evy process such that $\p_x(X_0=x)=1$ and define,
\[\tau=\inf\{t\ge0:X_t\le0\}<\infty.\]
Then the law of the process $Z$ is obtained by resurrecting under $\p_x$, $x\ge0$, the killed L\'evy process
\begin{equation}\label{7322}
Y_t=X_t\ind_{\{t<\tau\}}\,,\;\;t\ge0\,,
\end{equation}
according to the resurrection kernel,
\begin{equation}\label{eq:res-kernel}
K(\omega,dy)=\delta_{X_{\tau-}(\omega)}(dy).
\end{equation}
(See Subsection \ref{5224} for more details.) From Theorem 1.1 in \cite{inw} and Th\'eor\`eme 1 in \cite{me}, 
$Z$ is a strong Markov process with state space $[0,\infty)$, in a filtration where $(\tau_{n}$, $n\ge1)$ are 
stopping times. Note that $Y=(Y_t)_{t\ge0}$ is also a strong Markov process with state space $[0,\infty)$ and 
0 as an absorbing state. We will keep the notation $(\p_x)_{x\in[0,\infty)}$ for the family of probability 
measures associated to the process $Y$, and  $(P_x)_{x\in[0,\infty)}$ will denote the family of probability 
measures associated to $Z$. From our construction, when $\zeta$ is finite, $Z$ reaches 0 continuously, that 
is, for all $x\ge0$,
\[Z_{\zeta-}=0\,,\;\;\mbox{$P_x$-a.s.~on the set $\{\zeta<\infty\}$}.\]
So, either $Z$ reaches $0$ continuously at a finite time or it never reaches 0. In all the cases treated 
here, these events have probability 0 or 1, independently of $x$. The major part of this paper is 
devoted to determine conditions on the characteristics of the L\'evy process $X$ for $\zeta$ to be finite.\\

\subsection{The distribution of the resurrected process}\label{5222}
Let us now describe the distribution of the process $Z$ in more detail. Note that $\tau_1=\tau$, $\p_x$-a.s.,
for all $x\ge0$. Then it follows from (\ref{8633}) and (\ref{7322}) that for all nonnegative measurable 
function $f$ such that $f(0)=0$, $f(Z_t)=f(Y_t)+\sum_{n\ge2} f(X^{(n)}_t)\ind_{\{\tau_{n-1}\le t<\tau_{n}\}}$, 
$\p_x$-a.s., for all $t\ge0$, so that $Y$ is a subprocess of $Z$ in the sense of part III.3 in \cite{bg}, that 
is $\e_x(f(Y_t))\le E_x(f(Z_t))$, for all $t\ge0$ and $x\ge0$. This implies the existence of a 
multiplicative functional $M=(M_t)_{t\ge0}$ of $Z$ such that for all $t$, $x$ and $f$ as above, 
\begin{equation}\label{2256}
\e_x(f(Y_t))=E_x(f(Z_t)M_t)\,,
\end{equation}
see Theorem 2.3, p.101 in \cite{bg}. It also suggests that the distribution of $Z_t$ can be expressed from 
the process $Y$, at least in a non formal way, as follows 
\begin{equation}\label{2257}
E_x(f(Z_t))=\e_x(f(Y_t)M^{-1}_t(Y))\,.
\end{equation} 
The aim of Theorem \ref{7366} below is to make the functional $M^{-1}_t(Y)$ explicit and to give a direct 
proof of identity (\ref{2257}).\\

For that end, we next quote a result describing the joint distribution of $(\tau, X_{\tau-})$ under $\p_x$, 
for $x>0$. This is more general than needed right now, but it will be handy all over our work.  We denote by 
$\pi$ the L\'evy measure of $X$ and we set $\bar{\pi}^{-}(x)=\pi((-\infty,x])$, $x\in\mathbb{R}$. We will 
also denote by
\begin{equation}\label{pot0}
\e_x\left(\ind_{\{X_s\in dy, s<\tau\}}\right)ds=U^{0}(x; ds,dy),\qquad s,x,y\geq 0,
\end{equation}
the potential measure, in time and space, of $X$ killed at its first passage time below $0$
(that is the process $Y$). By $U$ and $U^*$, we denote the renewal measure of the bivariate ascending, 
respectively descending, ladder time and height process associated to $X$, see Chap.~VI in \cite{be}. 
The Wiener-Hopf factorization in time and space ensures that for any non negative and measurable function $h$,
\begin{equation*}
\begin{split}
&\iint_{[0,\infty)\times(0,\infty)}U^{0}(x; ds, dy)h(s,y)\\
&=\int_{[0,\infty)\times[0,x)}U^{*}(ds, d\ell)\int_{(0,\infty)\times[0,\infty)}
U(du, dv)h(s+u, x-\ell+v)\ind_{\{x-\ell+v>0\}},
\end{split}
\end{equation*}
where $U^*(ds,dy)=\delta_{(0,0)}(ds,dy)$ (resp. $U(ds,dy)=\delta_{(0,0)}(ds,dy)$) if $X$ (resp. $-X$) is 
a subordinator.

\begin{lemma}\label{RK-det} The joint distribution of $(\tau, X_{\tau-})$ is characterized by the following 
identity, which holds for any Borel function $h:\mathbb{R}_{+}\times\mathbb{R}_{+}\mapsto\mathbb{R}_{+}$, 
\begin{equation}\label{3458}
\begin{split}
&\iint_{[0,\infty)\times[0,\infty)}\p_x\left(\tau\in d t, X_{\tau-}\in dy, \tau<\infty\right)h(t,y)\\
=&\int_{[0,\infty)}a^{*}u^{*}(ds, x)h(s,0)\\
&+ \iint_{[0,\infty)\times(0,\infty)}U^{0}(x; ds, dy)h(s,y)\overline{\pi}^{-}(-y),
\end{split}
\end{equation}
where $a^{*}$ is the drift coefficient of the descending ladder height process of $X$ and $u^{*}(ds, x)$ 
denotes the density in the spatial coordinate of $U^{*}(ds, d\ell)$, which exists when $a^{*}>0$. 

For all $t>0$, $x>0$ and for all positive and measurable function 
$f:\mathbb{R}_{+}\times\mathbb{R}_{+}\mapsto\mathbb{R}_{+}$, 
\begin{eqnarray}\label{non-creapingprob}
\e_x(f(X_{\tau-},\tau)\ind_{\{\tau\le t,\,X_{\tau}<X_{\tau-}\}})&=&
\int_0^t\e_x(f(X_s,s)\bar{\pi}^{-}(-X_s)\ind_{\{s\le\tau\}})\,ds\,,\label{3499}
\end{eqnarray}
and 
\begin{equation}
\p_x(\tau\le t,\,X_{\tau-}=0)=\int_{(0,t]}a^{*}u^{*}(ds, x).\label{3498}
\end{equation}

Moreover, if $X$ drifts to $\infty$, then
$$\p_x(\tau=\infty)=\kappa^{*}U([0,\infty)\times[0,x]),\qquad x>0;$$ if $X$ drifts to $-\infty$, then 
$$\e_x(\tau)=\kappa^{-1} U^{*}([0,\infty)\times[0,x]),\qquad x>0;$$ while, if $X$ oscillates, then
$$\p_x(\tau=\infty)=0\quad\mbox{and}\quad \e_x(\tau)=\infty,\quad\mbox{for all $x>0$.}$$ 
Here $\kappa$ and $\kappa^{*}$ are the killing rates of the upward and downward ladder height processes, 
respectively.
\end{lemma}
\begin{proof}
The proof of identity~\eqref{3458} can be found in Theorem 3.1 of \cite{GM}, for the creeping part
and in Lemma 11 of \cite{dr}, for the jump part. The proof of (\ref{3499}) follows that of Lemma 11 of \cite{dr} 
up to a slight extension from the case $f\equiv1$ to the general case. The rest of the identities comes from 
Proposition 17 in Chapter VI of~\cite{be}.
\end{proof}

\begin{theorem}\label{9366}
For all $x\ge0$, $t\ge0$ and for all positive measurable function $f$, 
\begin{equation}\label{7366}
E_x(f(Z_t)\ind_{\{t<\zeta\}})=\e_x\left(f(X_t)\exp\left(\int_0^t\bar{\pi}^-(-X_s)\,ds\right)\ind_{\{t<\tau\}}\right)\,.
\end{equation}
\end{theorem}
\begin{proof} First note that identity (\ref{7366}) is trivial for $x=0$. Moreover, since {$(\tau_{n}, n\ge1)$} is 
a non decreasing sequence which satisfies $\zeta=\lim_{n\rightarrow\infty}\tau_{n}$, it suffices to show that for 
all $x>0$, $t\ge0$ and $n\ge1$, 
\begin{equation}\label{7334}
(n-1)!\,E_x\left(f(Z_t)\ind_{\{\tau_{n-1}\le t<\tau_{n}\}}\right)=
\e_x\left(f(X_t)\left(\int_0^t\bar{\pi}^-(-X_s)\,ds\right)^{n-1}\ind_{\{t<\tau\}}\right)\,.
\end{equation}
For $n=1$, the equality is trivial for all $x>0$ and $t\ge0$ (recall that $\tau_0=0$ and $\tau_1=\tau$, $\p_x$-a.s.). 
Then let us prove (\ref{7334}) by induction. Recall that {$(\tau_{n}, n\ge1)$} are stopping times in a 
filtration making $Z$ a strong Markov process. Moreover, they satisfy 
$\tau_{n}=\tau_1+\tau_{n-1}\circ\theta_{\tau_1}$.\\

Then let us fix $x>0$ and $n\ge2$. Assume that (\ref{7334}) holds for $n-1$ and for all $t>0$, 
and apply the strong Markov property at time $\tau_1$ in order to obtain,
\begin{eqnarray*}
&&E_x\left(f(Z_t)\ind_{\{\tau_{n-1}\le t<\tau_{n}\}}\right)=
E_x\left(f(Z_{t})\ind_{\{\tau_{n-1}\le t<\tau_{n},\,Z(\tau_{n-1})>0\}}\right)\\
&=&E_x\left(\ind_{\{\tau_1\le t,\,Z({\tau_1})>0\}}
\left(\ind_{\{\tau_{n-2}\le t-\tau_1<\tau_{n-1},\,
Z(\tau_{n-2})>0\}}f(Z_{t-\tau_1})\right)\circ\theta_{\tau_1}\right)\\
&=&E_x\left(\ind_{\{\tilde{\tau}_1\le t,\,\tilde{Z}({\tilde{\tau}_1})>0\}}
E_{\tilde{Z}({\tilde{\tau}_1})}\left(\ind_{\{\tau_{n-2}\le t-s<\tau_{n-1},\,
Z(\tau_{n-2})>0\}}f(Z_{t-s})\right)_{|{s=\tilde{\tau}_1}}\right),
\end{eqnarray*}
where in the last equality $(\tilde{\tau}_1,\tilde{Z}(\tilde{\tau}_1))$ is integrated under 
$P_x$ and has the same law as $(\tau_1,Z(\tau_1))$. Then by applying successively (\ref{3499}) in the 
second equality below and our induction hypothesis in the third one, we obtain,
\begin{eqnarray*}
&&E_x\left(\ind_{\{\tilde{\tau}_1\le t,\,\tilde{Z}({\tilde{\tau}_1})>0\}}
E_{\tilde{Z}({\tilde{\tau}_1})}\left(\ind_{\{\tau_{n-2}\le t-s<\tau_{n-1},\,
Z(\tau_{n-2})>0\}}f(Z_{t-s})\right)_{|{s=\tilde{\tau}_1}}\right)\\
&=&\e_x\left(\ind_{\{\tilde{\tau}\le t,\,\tilde{X}({\tilde{\tau}-})>0\}}
E_{\tilde{X}({\tilde{\tau}}-)}\left(\ind_{\{\tau_{n-2}\le t-s<\tau_{n-1},\,
Z(\tau_{n-2})>0\}}f(Z_{t-s})\right)_{|{s=\tilde{\tau}}}\right)\\
&=&\int_0^t\e_x\left(\bar{\pi}^{-}(-X_s)\ind_{\{s\le \tau\}}
E_{X_s}\left(\ind_{\{\tau_{n-2}\le t-s<\tau_{n-1},\,
Z(\tau_{n-2})>0\}}f(Z_{t-s})\right)\right)\,ds\\
&=&\frac1{(n-2)!}\int_0^t\e_x\left(\bar{\pi}^{-}(-X_s)\ind_{\{s\le\tau\}}\right.\\
&&\qquad\qquad\qquad\qquad\quad\times
\left.\e_{X_s}\left(f(X_{t-s})\left(\int_0^{t-s}\bar{\pi}^-(-X_u)\,du\right)^{n-2}
\ind_{\{t-s<\tau\}}\right)\right)ds\\
&=&\frac1{(n-2)!}\e_x\left(\int_{0}^t\bar{\pi}^-(-X_s)
\left(\int_s^{t}\bar{\pi}^-(-X_u)\,du\right)^{n-2}\,ds\,f(X_{t})\ind_{\{t<\tau\}}\right)\\
&=&\frac1{(n-1)!}\e_x\left(\left(\int_{0}^t
\bar{\pi}^-(-X_s)\,du\right)^{n-1}f(X_{t})\ind_{\{t<\tau\}}\right),\\
\end{eqnarray*}
where in the first equality, $(\tilde{\tau},\tilde{X}({\tilde{\tau}}-))$ is integrated under $\p_x$ and has 
the same law as $(\tau,{X}({{\tau}}-))$. This shows (\ref{7334}) and ends the proof of (\ref{7366}).
\end{proof}

\noindent Theorem \ref{9366}, up to a few additional justifications, shows that the multiplicative 
functional of $Z$ involved in (\ref{2256}) has the following expression,
\[\e_x(f(Y_t))=E_x\left(f(Z_t)\exp\left(-\int_0^t\bar{\pi}^-(-Z_s)\,ds\right)\ind_{\{t<\zeta\}}\right).\]
However, the interest of our work lies mainly in the identity  (\ref{7366}) which describes 
the law of $Z$ in terms of that of $X$.

\subsection{The resurrection kernel}\label{5224}
Let us define the kernel,  
\begin{equation}\label{QL}
\mathcal{K}_{\lambda}(x,dy)
=E_x\left(e^{-\lambda \tau_1}\ind_{\{Z_{\tau_1}\in dy,\,\tau_1<\infty\}}\right),\quad x,y,\lambda\ge0\,,
\end{equation}
and the function $f_\zeta^{(\lambda)}(x)=E_x\left(e^{-\lambda \zeta}\right)$, for $x\ge0$ and $\lambda\ge0$. 
Then we first note that $f_\zeta^{(\lambda)}$ is invariant for $\mathcal{K}_{\lambda}$. We will set 
$\mathcal{K}:=\mathcal{K}_{0}$ and $f_\zeta:=f_\zeta^{(0)}$.

\begin{proposition} For all $x\ge0$ and $\lambda\ge0$, 
\begin{equation}\label{invariance}
\mathcal{K}_{\lambda}f_\zeta^{(\lambda)}(x)=f_\zeta^{(\lambda)}(x)\,.
\end{equation}
In particular, the function $f_\zeta(x)=P_x(\zeta<\infty)$, $x\ge0$, satisfies,
\begin{equation}\label{4884}
\mathcal{K}f_\zeta(x)=f_\zeta(x)\,.
\end{equation}
\end{proposition}
\begin{proof}
From the strong Markov property applied for $Z$ at time $\tau_1$ and the identity 
$\zeta=\tau_1+\zeta\circ\theta_{\tau_1}$, we obtain
\begin{eqnarray*}
f_\zeta^{(\lambda)}(x)&=&E_x(\ind_{\{\tau_1<\infty\}}e^{-\lambda \tau_1}\e_{Z_{\tau_1}}(e^{-\lambda \zeta}))\nonumber\\
&=&\int_{y\in[0,\infty)}\mathcal{K}_{\lambda}(x,dy)f_{\zeta}^{(\lambda)}(y)\,,
\end{eqnarray*}
which proves (\ref{invariance}). Then (\ref{4884}) is obtained by taking $\lambda=0$.
\end{proof}

\noindent From (\ref{3458}) in Lemma \ref{RK-det}, the kernel $\mathcal{K}_{\lambda}$ can be made explicit 
as follows, 
\begin{eqnarray}
\mathcal{K}_{\lambda}(x,dy)&=&\e_x\left(e^{-\lambda \tau}\ind_{\{X_{\tau-}\in dy,\,\tau<\infty\}}\right)\nonumber\\
&=&U_\lambda^{0}(x,dy)\overline{\pi}^{-}(-y)\ind_{\{y>0\}}+a^*u_\lambda^*(x)\delta_0(dy),\label{5633}
\end{eqnarray}
where $U^{0}_\lambda(x,dy)=\int_0^\infty e^{-\lambda s}\p_x(X_s\in dy,\,s<\tau)\,ds$ is the $\lambda$-potential 
measure of the killed process $Y$ defined in the previous subsection, $u_\lambda^*(x)$ is the density of the 
$\lambda$-potential of the downward ladder height process of $X$ and $a^*$ is its drift coefficient.\\ 

Following our objective, we wish to obtain more information on the function 
$f_\zeta(x)=P_x(\zeta<\infty)$. Note that when the process $X$ drifts toward $\infty$, 
$\p_x(\tau<\infty)=P_x(\tau_1<\infty)<1$ and since $\zeta\ge\tau_1$, $P_x$-a.s., we have $f_\zeta(x)<1$. 
On the other hand, when the process $X$ does not drift toward $\infty$, we have 
$\p_x(\tau<\infty)=P_x(\tau_1<\infty)=1$ so that the kernel 
\[\mathcal{K}(x,dy)=U^{0}(x,dy)\overline{\pi}^{-}(-y)\ind_{\{y>0\}}+a^*u^*(x)\delta_0(dy)\] 
is Markovian. It completes the description of the resurrection kernel given in (\ref{eq:res-kernel}). 
It is actually the transition kernel of the Markov chain $(Z_{\tau_n})_{n\ge0}$, that is for all bounded Borel 
function $f$, 
\[\mathcal{K}^{(n)}f(x)=E_x(f(Z_{\tau_n})),\] 
where $\mathcal{K}^{(n)}$ denotes the $n$-th composition of $\mathcal{K}$ with itself.
Equation~\eqref{4884} tells us that the function $f_{\zeta}$ is a bounded invariant function for 
$\mathcal{K}.$ In our forthcoming analysis of cases, we will encounter that a zero-one law arises, either 
$f_{\zeta}\equiv 1$ or $f_{\zeta}\equiv 0.$ From there, we conjecture that 
\begin{itemize}\item{\it If $X$ does not drift towards $\infty$, then either $f_{\zeta}(x)=1,$ for all 
$x\in(0,\infty),$ or $f_{\zeta}(x)=0,$ for all $x\in(0,\infty)$.}
\end{itemize}A possible approach to prove this conjecture would require studying the totality of invariant 
functions for the Markovian kernel $\mathcal{K}.$ In particular, if one is able to prove that the totality 
of bounded invariant functions for $\mathcal{K}$ are the constant functions, then necessarily $f_{\zeta}$ 
would be a constant function, and hence equal to $0$ or $1$. We invite the interested reader to prove or 
disprove this conjecture.

\section{The creeping case}\label{6362}

Let us start with the case where the L\'evy process $X$ creeps downward. Recall that by definition,
this means that for all $x>0$,
\[\p_x(X_{\tau-}=0,\,\tau<\infty)=a^*u^*(x)>0\] and that $X$ creeps downward if and only if the 
drift $a^*$ is positive. Moreover, in this case, $u^*$ is continuous on $[0,\infty)$ and satisfies
$\lim_{x\rightarrow0+}a^*u^*(x)=1$. 

\begin{theorem}\label{9282}
Assume that $X$ creeps downward and that it does not drift toward $\infty$. 
If either $-X$ is a subordinator or if the downward ladder height process of $X$ has finite mean, then 
$f_\zeta(x)=P_x(\zeta<\infty)=1$, for all $x\ge0$. 
\end{theorem}
\begin{proof}
First observe that since $X$ does not drift toward $\infty$, $P_x(\tau_n<\infty)=1$, for all $n\ge1$. 
In particular $\p_x(\tau<\infty)=1$ and $\p_x(X_{\tau-}=0)=P_x(Z_{\tau_1}=0)=a^*u^*(x)$.
Then from the Markov property and the identity $\tau_n=\tau_1+\tau_{n-1}\circ\theta_{\tau_1}$, 
\begin{eqnarray}
P_x(Z_{\tau_n}>0)&=&E_x(\ind_{\{Z_{\tau_{n-1}}>0\}}P_{Z_{\tau_{n-1}}}(Z_{\tau_1}>0))\nonumber\\
&=&E_x(\ind_{\{Z_{\tau_{n-1}}>0\}}[1-a^*u^*(Z_{\tau_{n-1}})]).\label{7533}
\end{eqnarray}
If $-X$ is a subordinator, then $Z_{\tau_{n}}\le x$, $P_x$-a.s., for all $n$, so that 
\[1-a^*u^*(Z_{\tau_{n-1}})\le 1-\inf_{y\in[0,x]}a^*u^*(y):=k<1,\quad P_x-a.s.,\]
and hence, from equality (\ref{7533}), $P_x(Z_{\tau_n}>0)\le kP_x(Z_{\tau_{n-1}}>0)$, which implies 
$P_x(Z_{\tau_n}>0)<k^n$. Then we derive from Borel-Cantelli lemma that $P_x$-a.s., $Z_{\tau_n}>0$ holds
only a finite number of times and therefore $Z$ is absorbed at a finite time. 

Let us now consider the case where $X$ does not drift toward $\infty$ and creeps downward, and assume that 
the downward ladder height process of $X$ has finite mean. Then recall from \cite{bvs} that 
$\lim_{y\rightarrow\infty}u^*(y)=a^*/m>0$, where $m$ is the mean of the downward ladder height 
process. This yields, 
\[1-a^*u^*(Z_{\tau_{n-1}})\le 1-\inf_{y\ge0}a^*u^*(y)<1,\quad P_x-a.s.,\]
and the same argument as above leads to the same conclusion. 
\end{proof}
\noindent In view of Theorem \ref{drift-infty} below, it seems that integrability of the downward ladder 
height process when $-X$ is not a subordinator is not a necessary condition in Theorem \ref{9282}.
However, although it is a little counterintuitive at first glance, it is possible that 'big' negative 
jumps of $X$ at its infimum play an important role for the conservativeness property of the resurrected 
process. 

\section{Some criteria for non absorption}\label{section:MC}

Let us denote by $\tau^{-}_{z},$ the first passage time below $z$ by $X$, that is
$$\tau^{-}_{z}=\inf\{t\ge0: X_t\le z\},\qquad z\in \mathbb{R}.$$
By our construction in Subsection \ref{5221} and from the strong Markov property, conditionally on the 
$n-1$ first positions where the process is resurrected, say $(Z(\tau_0)=x_0,\,Z(\tau_1)=x_{1},\ldots, 
Z(\tau_{n-1})=x_{n-1})$, the $n$-th resurrection time under $P_{x_0}$ has the same distribution as  
\begin{equation}\label{6367}
\tau_{n}=\sum^{n-1}_{i=0}\tau^{-,i}_{-x_{i}},\quad n\ge1,
\end{equation}
where $(\tau^{-,0}_{-x_{0}},\tau^{-,1}_{-x_{1}},\ldots, \tau^{-,n}_{-x_{n}},\ldots)$ are independent random 
variables, and the law of $\tau^{-,i}_{-z}$  is the same as that of $\tau^{-}_{-z}$ under $\p$. We deduce 
therefrom that, conditionally on the resurrection positions 
$(Z(\tau_0)=x_0,\,Z(\tau_1)=x_{1},\ldots, Z(\tau_n)=x_{n},\ldots)$,
the resurrected process $Z$ will be absorbed at $0$ in a finite time if and only if  
$\sum_{n\geq0}\tau^{-,n}_{-x_{n}}<\infty$.
This argument yields the following identity in law,
\begin{equation}\label{limz}
\zeta\stackrel{(\mbox{\tiny Law})}{=}\sum_{n\ge0}\tau^{-,n}_{-Z(\tau_n)},
\end{equation}  
where the family of processes $\{(\tau^{-,i}_{-x_{i}}),\,x_i\ge0,i\ge0\}$ is independent of the sequence 
$(Z(\tau_i))_{i\ge0}$. This identity will be fundamental in what follows. Another interesting identity is given in the 
following Lemma. This Lemma, together with the identities (\ref{6367}) and  (\ref{limz}) will be useful to establish 
Proposition \ref{infinite}. 
\begin{lemma}\label{exponentialdistr} Assume that the L\'evy process $X$ does not creep downward and does 
not drift to $\infty.$  Then the random variable 
$$\int^{\tau}_{0}\overline{\pi}^{-}(-X_{t})dt $$ 
is exponentially distributed with parameter $1$.
\end{lemma}
\noindent Before we state our next result, we would like to point out that a version of Lemma \ref{exponentialdistr} 
for a class of rather general  Markov processes can be found in \cite{sa}, see Proposition 2.7 therein. Note also 
that a direct proof of Lemma \ref{exponentialdistr} can be obtained, {arguing as in the proof of Theorem~\ref{7366},}
to identify the moments of the variable in 
question using the Kac's moment formula. For sake of brevity, we avoid the details. 

\begin{proposition}\label{infinite}
The resurrected process has an infinite lifetime, that is $P_x(\zeta=\infty)=1$, for all $x>0$ in either of the 
following cases:
\begin{itemize}
\item[$(a)$] $X$ does not creep downward and $0<\pi(-\infty,0)<\infty$.
\item[$(b)$] $0$ is not regular for $(-\infty,0)$.
\end{itemize} 
\end{proposition}
\begin{proof} Let us first make hypothesis $(a)$.
We assume first that $X$ does not drift towards $\infty$.
Then from the former lemma we have that  for $x>0,$ under $\p_{x}$ the random variable 
$\int^{\tau}_{0}\overline{\pi}^{-}(-X_{t})dt$ 
follows an exponential distribution with parameter $1.$ Moreover it satisfies the inequality 
$\int^{\tau}_{0}\overline{\pi}^{-}(-X_{t})dt\leq \overline{\pi}^{-}(0)\tau$.
Recall also that, in the notation of the beginning of this section, under $\p_x$, $\tau$ has the same 
law as $\tau_{-x}^-$ under $\p$. Then from the above observations and the representation (\ref{6367}) of 
$\tau_n$, we obtain the stochastic domination  
\begin{equation}\label{stoch-dom}
\tau_{n}\stackrel{\text{Law}}{\geq} \frac{1}{\overline{\pi}^{-}(0)}\sum^{n-1}_{i=0}\mathbf{e}_{i},
\end{equation} 
where  $\mathbf{e}_{i}$, $i\geq0$ are i.i.d.~standard exponential r.v.'s. It follows that $\tau_{n}\to \infty$, 
a.s.~when $n\to \infty$. Assume now that $X$ drifts towards $\infty$. In this case, we have that 
$P_{x}(\zeta=\infty)\geq \p_x(\tau=\infty)>0$, $x>0$. To prove that the former probability equals one, 
independently of the starting point, we will prove that the event $\{\zeta<\infty\}$ has zero probability. 
Since the process does not creep downward, the only way in which the resurrected process gets absorbed at 0, 
viz. $\zeta<\infty,$ is by infinite resurrections whose sum of lengths is finite. But this is impossible because, 
in this case, each resurrection time is stochastically bounded by below by an exponential random variable of 
parameter ${\pi}^{-}(0).$ Indeed, to have a resurrection, at least there should be a negative jump, which 
happens at an exponential time with parameter ${\pi}^{-}(0).$ This concludes the proof for case $(a)$.

Let us now make hypothesis $(b)$. If $0$ is not regular for the half-line $(-\infty,0)$, the downward ladder 
height process $\widehat{H}$ has a finite L\'evy measure and zero drift, so the process $X$ can not creep 
downwards. We assume for a moment that 
$X$ does not drift towards $\infty$, and hence $\widehat{H}$ has also an infinite lifetime. The case $(a),$ 
proved above, ensures that the process obtained by resurrection of $-\widehat{H}$ is never absorbed 
at zero. Since in the local time scale this process bounds by below the process $Z$, we infer that the latter 
is never absorbed at $0.$ 

Then we deal with the case where $X$ drifts towards $\infty.$ As in the proof of case $(a)$, we see that the 
event $\{\zeta<\infty\},$ has zero probability. Indeed, in the local time scale, the downward ladder height 
process is never absorbed at zero, which bounds the resurrected process from below, and then an infinite 
excursion from the infimum, inside the latest resurrection, arises, and from there onwards there is no need 
to resurrect the process again as it never goes below the reached infimum. 
\end{proof}


\section{When $X$ drifts towards $-\infty$}\label{8362}

Let us recall the notations introduced before Lemma \ref{RK-det} in Section \ref{2456} and denote the renewal 
functions of the downward and upward ladder height processes respectively by 
$$U^{*}([0,\infty)\times[0,x]):=U^{*}(x)\,,\qquad U([0,\infty)\times[0,x]):=U(x),\qquad x>0.$$ 
Recall that $\kappa\in[0,\infty)$ is the killing rate of the upward ladder process, that is 
$U([0,\infty)\times[0,\infty))=\kappa^{-1}$. Moreover, from Lemma~\ref{RK-det}, when $X$ drifts towards $-\infty$,
$\kappa>0$ and
\begin{equation}\label{3522}
\e_x\left(\tau\right)=\kappa^{-1} U^{*}(x), \qquad x\ge0.
\end{equation}
In this section we give sufficient conditions for the lifetime of the resurrected process to be finite. 

\begin{theorem}\label{drift-infty} Assume that,

$(i)$ $0$ is regular for $(-\infty,0)$, 

$(ii)$ $X$ drifts towards $-\infty$, 

$(iii)$ the following condition is satisfied, 
\begin{equation}\label{H-inf}
\sup_{y>0}U^{*}(y)\overline{\pi}^{-}(-y)<\kappa.
\end{equation} 
Then $E_{x}(\zeta)<\infty,$ and in particular $P_x(\zeta<\infty)=1$, for all $x\ge0$.
\end{theorem}
\begin{remark}
It is worth pointing out that the constant $\kappa$ depends on the chosen normalization of the 
local time at the supremum, which is actually related to that chosen at the infimum, thus, 
changing it, would lead to a change the representation of $U^{*}.$  It could hence be taken as $1.$ 
\end{remark}

\begin{proof}
We derive from (\ref{limz}) that
\begin{equation} \label{3778}
E_x(\zeta)=\sum_{n\ge0}E_x\left(\tau^{-,n}_{-Z(\tau_n)}\right),
\end{equation}  
where we recall that, under $P_x$, the family of processes $\{(\tau^{-,n}_{-x_{n}}),\,x_n\ge0,n\ge0\}$ is 
independent of the sequence $(Z(\tau_n))_{n\ge0}$. Moreover, under $P_x$, the variables $\tau^{-,n}_{-x_{n}}$, 
$x_n\ge0$, $n\ge0$ are independent and $\tau^{-,n}_{-x_{n}}$ has the same law as 
$\inf\{t\ge0:X_t\le-x_n\}$ under $\p$. Note that from assumption $(ii)$ of the statement, $\tau_n<\infty$, 
$P_x$-a.s. for all $x\ge0$. Then from (\ref{3522}), the relation 
$\tau_{n}=\tau_{n-1}+\tau_1\circ\theta_{\tau_{n-1}}$ and the strong Markov property applied at time 
$\tau_{n-1}$, we obtain that for all $n\ge1$,
\begin{eqnarray*}
E_x\left(\tau^{-,n}_{-Z(\tau_n)}\right)&=&\kappa^{-1} E_x\left(U^*(Z(\tau_n))\right)\\
&=&\kappa^{-1} E_x\left(E_{Z(\tau_{n-1})}\left(U^*(Z(\tau_1)\right)\right)\,.
\end{eqnarray*}
Since $X$ drifts towards $-\infty$, we kave $\kappa>0$. Then it follows from (\ref{3458}) in Lemma 
\ref{RK-det} that 
\begin{eqnarray*}
&&\kappa^{-1} E_x\left(E_{Z(\tau_{n-1})}\left(U^*(Z(\tau_1)\right)\right)=
\kappa^{-1} E_x\left(\e_{Z(\tau_{n-1})}\left(U^*(X(\tau-)\right)\right)\\
&=&\kappa^{-1} E_x\left(\int^{Z(\tau_{n-1})}_{0}U^{*}(dy)\int^{\infty}_0U(dz) U^{*}(Z(\tau_{n-1})-y+z)
\overline{\pi}^{-}(y-Z(\tau_{n-1})-z)\right)\\
&&+a^*U^*(0)E_x(u^{*}(Z(\tau_{n-1}))\\
&\leq&\kappa U([0,\infty)\times[0,\infty))\sup_{y>0}U^{*}(y)\overline{\pi}^{-}(-y)
E_x\left(U^*(Z(\tau_{n-1}))\right)+a^*U^*(0)E_x(u^{*}(Z(\tau_{n-1}))\\
&=&\kappa^{-2}\sup_{y>0}U^{*}(y)\overline{\pi}^{-}(-y)E_x\left(U^*(Z(\tau_{n-1}))\right)=
\kappa^{-1}\sup_{y>0}U^{*}(y)\overline{\pi}^{-}(-y)E_x\left(\tau^{-,n-1}_{-Z(\tau_{n-1})}\right).
\end{eqnarray*}
Note that $U^*(0)=0$ since 0 is regular for $(-\infty,0)$.
It follows from the above inequalities that for all $n\ge1$, $E_x\left(\tau^{-,n}_{-Z(\tau_n)}\right)\le c 
E_x\left(\tau^{-,n-1}_{-Z(\tau_{n-1})}\right)$, where $c:=\kappa^{-1}\sup_{y>0}U^{*}(y)\overline{\pi}^{-}(-y)$ 
and hence {$E_x\left(\tau^{-,n}_{-Z(\tau_n)}\right)\le c^{n}\e_x\left(\tau\right)$}.
Together with (\ref{H-inf}) and (\ref{3778}), this implies that for any $x>0,$ $\e_{x}\left(\zeta\right)\leq 
\frac{1}{1-c}{\frac{1}{k}U^*(x)=\frac{1}{1-c}\e_x\left(\tau\right)<\infty}.$ 
\end{proof}
\noindent We emphasize that {downwards} creeping L\'evy processes always satisfy condition $(i)$ of Theorem 
\ref{drift-infty} and when they also satisfy conditions of Theorem \ref{9282}, then conditions $(ii)$ 
and $(iii)$ of Theorem \ref{drift-infty}  are not needed for the associated resurrected process to be 
absorbed at 0 in a finite time.\\
 
For the remainder of this section, we will focus on the special case where $X$ is a non  increasing L\'evy 
process, that is the negative of a subordinator. Condition (\ref{H-inf}) can be very useful in this 
particular case. Indeed, when $X$ is decreasing, $U^*$ is the renewal function of the process $X$ itself. 
In this case, $U^*$ will be written as, $U^*(x)=\int_0^\infty\p(0\le -X_t\le x)\,dt$.
Moreover, recall that the renewal measure $U(dt,dx)$ has the simple form $U(dt,dx)=\delta_{(0,0)}(dt,dx)$, 
so that $\kappa=1$. Then conditions of Theorem \ref{drift-infty} are satisfied whenever $X$ has no negative 
drift, $\pi(-\epsilon,0)=\infty$, for all $\epsilon>0$ and (\ref{H-inf}) holds and we obtain the 
following corollary. 

\begin{corollary}\label{3729}
Assume that $X$ is the negative of a subordinator with no $($negative$)$ drift and such that 
$\pi(-\epsilon,0)=\infty$, for all $\epsilon>0$. If
\begin{equation}\label{7224}
\sup_{y>0}U^{*}(y)\overline{\pi}^{-}(-y)<1,
\end{equation}
then $P_x(\zeta<\infty)=1$, for all $x\ge0$.
\end{corollary}

\noindent This result leads to the following corollary that covers many commonly found examples of 
subordinators.
\begin{corollary}\label{3789}
Assume that $X$ is the negative of a subordinator whose tail L\'evy measure satisfies that there are 
$0<\alpha,\beta<1,$ such that $\overline{\pi}^{-}$ is regularly varying at $0$ with index $\alpha$ 
and at infinity with index $\beta,$ viz. for all $c>0,$
$$\lim_{x\to 0+}\frac{\overline{\pi}^{-}(-xc)}{\overline{\pi}^{-}(-x)}={c}^{-\alpha},\qquad 
\lim_{x\to \infty}\frac{\overline{\pi}^{-}(-xc)}{\overline{\pi}^{-}(-x)}={c}^{-\beta}.$$
In this case, we have that $P_x(\zeta<\infty)=1$, for all $x\ge0$.
\end{corollary}
\begin{proof}
From the estimates in page 75 in \cite{be} and the reflection formula for the Gamma function, we know that 
under the assumptions of the Corollary we have 
$$\lim_{y\to 0}U^{*}(y)\overline{\pi}^{-}(-y)=\frac{\sin{\pi\alpha}}{\pi\alpha},\qquad 
\lim_{y\to 0}U^{*}(y)\overline{\pi}^{-}(-y)=\frac{\sin{\pi\beta}}{\pi\beta}.$$ The latter are both valued in 
$(0,1).$ It follows from Theorem 3 in \cite{dk} that for all $x>0$, 
$\p(X_{\tau^{-}_{-x}-}-X_{\tau^{-}_{-x}}>x)=U^{*}(x)\overline{\pi}^{-}(-x)$. The assumptions imply that 
$\overline{\pi}^{-}(-x)>0,$ $\forall x>0,$ and hence the latter probability is necessarily in $(0,1).$ 
We can conclude from here that the condition \eqref{H-inf} is satisfied. \end{proof}

\noindent Let us quote two recent works related to the two above corollaries. First of all, the case of 
stable subordinators was treated in \cite{dtx}, see also \cite{ksv} and the following section. Furthermore, 
in the very recent work \cite{li}, the author treats in Chapter 4 the more general case of subordinators
whose density is completely monotone, which, added to a few other hypotheses, is however more restrictive 
than condition (\ref{7224}).

\begin{remark}
The previous corollary suggest a method to build examples of L\'evy process for which 
\begin{equation}\label{2554}
\sup_{x>0}U^{*}(x)\overline{\pi}^{-}(-x)=1.
\end{equation}
For instance, this is the case when $-X$ is a Gamma subordinator,  that is 
\[\e(\exp({\lambda X_1}))=(1+\lambda/b)^{-a}=
\exp\left(-\int_0^\infty(1-e^{-\lambda x})ax^{-1}e^{-bx}\,dx\right)\,,\quad a,b,\lambda>0,\] 
see \cite{be} p.75. In this interesting setting, Theorem~\ref{drift-infty} is not conclusive and other 
techniques seem to be necessary.
\end{remark}

\section{The stable case}\label{5223}

The case where $X$ is a stable L\'evy process is very particular as our problem can be tackled and entirely 
solved by using the Lamperti transformation. The method we will develop in this section was the starting point 
of our research a long time ago. The same method has recently been used in \cite{ksv} to show this result 
in the more general framework of positive self-similar Markov processes. This motivated us to get back to our 
work and produce the present note.

We consider $X$ a stable L\'evy process with index $\alpha\in(0,2]$ and recall from (\ref{7322}) the definition 
of the killed process $Y$. In this setting, both Markov processes $Y$ and $Z$ clearly inherit 
from $X$ the scaling property of index $1/\alpha$. As positive self-similar Markov processes, $Y$ and $Z$ can 
each be represented as the exponential of some possibly killed L\'evy process, time changed by the inverse of 
its exponential functional, see \cite{la}. Let $\xi^{Y}$ (resp. $\xi^Z$) be the underlying L\'evy process in 
the Lamperti representation of $Y$, (resp. $Z$). Then our construction of $Z$ from $Y$ and the Lamperti 
representation of both processes show that $\xi^{Y}$ is obtained from (the non killed L\'evy process) $\xi^Z$ 
by killing it at an independent exponential time of some parameter, say $\beta\ge0$. Note that $\beta=0$ if 
and only if $X$ has no negative jumps. In this case, $Z=Y$ and the latter process clearly hits 0 in a finite 
time almost surely. Therefore, we can assume that $\beta>0$.

Using these arguments and standard facts from the theory of self-similar Markov processes, we obtain that $Z$ 
hits $0$ in a finite time if and only if $\xi^Z$ drifts towards $-\infty.$ On the other hand, we know from 
\cite{cc} and \cite{kp} that the process $\xi^{Y}$ is a process of the hypergeometric type, with 
characteristic exponent,
$$\e\left(e^{i\lambda\xi^{Y}_{1}},1<{\rm e}\right)=\exp\{-\Psi(\lambda)\}\quad\mbox{so that}\quad
\e\left(e^{i\lambda\xi^{Z}_{1}}\right)=\exp\{-(\Psi(\lambda)-\Psi(0))\},$$ 
where $\displaystyle\Psi(\lambda)= 
\frac{\Gamma\left(\alpha-i\lambda\right)}
{\Gamma\left(\alpha\overline{\rho}-i\lambda\right)}\frac{\Gamma\left(1+i\lambda\right)}
{\Gamma\left(1-\alpha\overline{\rho}+i\lambda\right)}$, with $\overline{\rho}=\p(X_1<0)$ and where ${\rm e}$ 
denotes the lifetime of $\xi^Y$. It is then easily verified that there is a constant $C_{\alpha}>0$ such that,
\begin{equation}\label{2369}
\e\left(\xi_{1}^Z\right)=\e\left(\xi^{Y}_{1}\,|\,1<{\rm e}\right)=C_{\alpha}\left(\left(\psi(1-
\alpha\overline{\rho})-\psi(1)\right)-\left(\psi(\alpha\overline{\rho})-\psi(\alpha)\right)\right),
\end{equation}
where $\psi(\beta)$ denotes the digamma function 
$\psi(\beta)=\frac{\Gamma^{\prime}(\beta)}{\Gamma(\beta)}$. We are now able to solve the problem of the 
finiteness of the lifetime of $Z$ in the stable case. 

\begin{theorem}\label{8344}
Assume that $X$ is a stable L\'evy process with index $\alpha\in(0,2]$. 

Then $P_x(\zeta<\infty)=1$, for all $x\ge0$ if and only if $\alpha$ and $\overline{\rho}$ satisfy,
\begin{equation}\label{9852}
\cot(\pi\alpha\overline{\rho})<\int^{\infty}_{0}\frac{dt}{1-e^{-t}}
\left(e^{-\alpha t}-e^{-t}\right).
\end{equation}
\end{theorem}
\begin{proof} As argued before the statement of the theorem, $P_x(\zeta<\infty)=1$, for all $x\ge0$ if 
and only if $\xi^Z$ drifts towards $-\infty$, which is equivalent to $\e\left(\xi_{1}^Z\right)<0$. 
From (\ref{2369}) we are then left to find the values $\alpha$ and $\overline{\rho}$ such that 
$\psi(1-\alpha\overline{\rho})-\psi(1)-\psi(\alpha\overline{\rho})+\psi(\alpha)<0$.

Then note that by the reflection formula for the digamma function, 
$$\psi(1-\alpha\overline{\rho})-\psi(\alpha\overline{\rho})=\cot(\pi\alpha\overline{\rho}).$$
On the other hand, the following identity for the digamma function is well known 
$$\psi(\delta)-\psi(\gamma)=\int^{\infty}_{0}\frac{dt}{1-e^{-t}}
\left(e^{-\gamma t}-e^{-\delta t}\right),\qquad \delta, \gamma\geq 0,$$ 
and this allows us to conclude.
\end{proof}
\noindent Since $\cot(\pi\alpha\overline{\rho})\ge0$ if $\alpha\overline{\rho}\in[0,1/2]$ and 
$\cot(\pi\alpha\overline{\rho})<0$ if $\alpha\overline{\rho}\in(1/2,1]$, it follows from (\ref{9852}), 
that if $\alpha<1$ and $\alpha\overline{\rho}>1/2$ then
$P_x(\zeta<\infty)=1$, for all $x\ge0$, whereas if $\alpha\overline{\rho}\leq 1/2$ 
and $\alpha\geq 1$ then $P_x(\zeta<\infty)=0$, for all $x>0$.
Note also that when $-X$ is a stable subordinator, Corollary \ref{3789} cannot be recovered from 
(\ref{9852}) without any further development. However, it can be derived directly from Lamperti's 
transformation. Indeed, in this case, $Z$ is a decreasing self-similar Markov process whose only 
alternative is to hit 0 through an accumulation of jumps in a finite time. 

\section{Open question and perspectives}\label{8849}

The aim of the present paper was to obtain necessary and sufficient conditions on a L\'evy process 
for its resurrected version to satisfy the following property:  
\[(P)\qquad\mbox{$P_x(\zeta<\infty)=1$ for all $x\in[0,\infty)$.}\]
The various open questions raised throughout this note actually boil down to the fact that some 
class of L\'evy processes resists our investigations. When the initial L\'evy process $X$ either 
drifts toward infinity or when 0 is not regular for $(-\infty,0)$, the property $(P)$ fails, as 
shown in the comment at the end of Subsection \ref{5224} and in Proposition \ref{infinite}. On 
the other hand, the stable case in Theorem \ref{8344} shows that it is not enough that $X$ does 
not drift towards infinity and that 0 is regular for $(-\infty,0)$ for the property $(P)$ to be 
satisfied. In summary, the L\'evy processes $X$ for which work remains to be done are those that 
verify the following properties: 
\begin{itemize}
\item[$\bullet$] $X$ does not drift to $\infty$,
\item[$\bullet$] 0 is regular for $(-\infty,0)$,
\item[$\bullet$] $X$ falls outside 
the cases covered by Theorems \ref{9282}, \ref{drift-infty}  and \ref{8344}.
\end{itemize}


Let us finally emphasize that some of our results can be extended to general $\mathbb{R}^d$-valued 
Markov processes. We consider the resurrection of such a process when leaving an open subset 
$D\subset\mathbb{R}^d$. The rate function $x\mapsto \bar{\pi}^-(-x)$ involved in the killing of the 
process is then given by $x\mapsto N(x,\ind_{D^c})$, where $N(x,dy)$ is the kernel of the L\'evy 
system of the process and Theorem \ref{9366} remains valid where the multiplicative functional 
$\exp\left(\int_0^tN(X_s,\ind_{D^c})\,ds\right)$ now defines the resurrected process. We can also
claim, as an extension of Theorem \ref{drift-infty}, that provided the condition 
$\sup_{x\in D}\e_x(\tau_{D^c})N(x,\ind_{D^c})<1$ is satisfied, where $\tau_{D^c}$ is the first 
exist time from $D$, the lifetime of the resurrected process has finite mean. These few extensions 
allow us to believe that other more refined results can be obtained in fairly general frameworks
thus offering some perspectives in this direction.

\vspace*{0.3in}

\noindent{\bf Acknowledgements.} This paper was concluded while V.R.~was visiting the Department of 
Statistics at the University of Warwick, United Kingdom; he would like thank his hosts for partial 
financial support as well as for their kindness and hospitality. In addition, V.R.~is grateful for 
additional financial support from  CONAHCyT-Mexico, grant nr. 852367. 
M.E.C.~{and V.R. would like to thank}~the kind hospitality offered by L.C.~and LAREMA {for the 
elaboration of this work.} L.C.~was supported by the ANR project "Rawabranch" number ANR-23-CE40-0008.
The authors are grateful to Zoran Vondra\v{c}ek for pointing out on the references \cite{dtx} and 
\cite{li} and to Pat Fitzsimmons for helpful discussions on the subject of this paper. 
{Last, but not least, we would like to thank the anonymous referee for {her/his} insightful report.}


\vspace*{0.3in}

\begin{thebibliography}{99}

\bibitem{be} \sc J.~Bertoin: \it L\'evy processes. \rm Cambridge University Press, Cambridge, 1996. 

\bibitem{bvs} \sc J.~Bertoin, K.~van Harn and  F.~Steutel: \rm Renewal theory and level passage by 
subordinators. {\it Statist. Probab. Lett.} 45, (1999), no.1, 65--69.

\bibitem{cb} \sc K.~Bogdan, K.~Burdzy and Z.-Q.~Chen: \rm Censored stable processes. 
{\it Probab. Theory Related Fields}, 127,  no. 1, 89--152, (2003). 

\bibitem{bg} \sc R.M.~Blumenthal and R.K.~Getoor: {\it Markov processes and potential theory.}
\rm Pure and Applied Mathematics, Vol. 29 Academic Press, New York-London 1968. 

\bibitem{cc} \sc M.E.~Caballero and L.~Chaumont: \rm  Conditioned stable L\'evy processes and the Lamperti 
representation. {\it J. Appl. Probab.} 43, no. 4, 967--983, (2006). 

\bibitem{dk} \sc R.A.~Doney and A.E.~Kyprianou: \rm Overshoots and undershoots of L\'evy processes. 
{\it Ann. Appl. Probab.} 16, no. 1, 91--106, (2006). 

\bibitem{dr} \sc R.A.~Doney and V.~Rivero: \rm  Asymptotic behaviour of first passage time distributions for 
L\'evy processes. {\it Probab. Theory Related Fields}, 157,  no. 1-2, 1--45, (2013).

\bibitem{dtx} \sc Q.~Du, L.~Toniazzi and Z.~Xu: \rm Censored stable subordinators and fractional derivatives
{\it Fract. Calc. Appl. Anal.} 24 (2021), no. 4, 1035--1068.

\bibitem{GM}\sc P.S.~Griffin and R.A.~Maller: \rm The time at which a Lévy process creeps 
{\it Electron. J. Probab.} 16, 2182--2202 (2011).

\bibitem{inw} \sc  N.~Ikeda, M.~Nagasawa and S.~Watanabe: \rm A construction of Markov processes by piecing 
out. {\it Proc. Japan Acad.} 42, 370--375, (1966).

\bibitem{ksv} \sc P.~Kim, R.~Song and Z.~Vondra\v{c}ek: \rm Positive self-similar Markov processes obtained by 
resurrection. {\it Stochastic Process. Appl.} 156 (2023), 379--420.

\bibitem{kp} \sc A.~Kuznetsov and J.C.~Pardo: \rm Fluctuations of stable processes and exponential 
functionals of hypergeometric L\'evy processes {\it Acta Appl. Math.} 123 (2013), 113--139.

\bibitem{la} \sc J.~Lamperti: \rm Semi-stable Markov processes. I {\it 
Z. Wahrscheinlichkeitstheorie und Verw. Gebiete} 22 (1972), 205--225.

\bibitem{li} \sc C.~Li: \rm Fractional Time Derivatives and Stochastic Processes
{\it PhD thesis, Technischen Universit\"at Dresden}, January 2024. 

\bibitem{me} \sc P.A.~Meyer: \rm Renaissance, recollements, m\'elanges, ralentissement de processus de Markov. 
{\it Ann. Inst. Fourier}, 25, no. 3-4, xxiii, 465--497, (1975). 

\bibitem{neveu} \sc J.~Neveu: \rm Martingales à temps discret. {\it Masson et Cie, Éditeurs, Paris,} 1972. 

\bibitem{sa} \sc S.~Sato: \rm On the reconstruction of a killed Markov process.
{\it S\'eminaire de probabilit\'es} (Strasbourg), tome 26 (1992), p. 540--559.

\bibitem{wa} \sc V.~Wagner: \rm Censored symmetric L\'evy-type processes. 
Forum Math. 31, no. 6, 1351--1368, (2019). 
\end{thebibliography}
\end{document}